\documentclass[11pt]{article}
\usepackage[dutch,english]{babel}
\usepackage{amssymb}
\usepackage{amsmath}
\usepackage{amsfonts}
\usepackage{amsthm}
\usepackage{accents} 
\usepackage{pgf}
\usepackage{tikz}
\usepackage{mathtools}
\usepackage{fancyhdr, layout}
\usepackage{changepage}
\usepackage{hyperref}
\usepackage{enumitem}
\usepackage{todonotes}
\usetikzlibrary{shadows}

\newtheorem{theorem}{Theorem}

\newtheorem{proposition}{Proposition}
\newtheorem{definition}{Definition}

\newtheorem{cor}{Corollary}

\newcommand{\ie}{i.e.,~}

\newcommand{\eqqref}[1]{Eq.~(\ref{#1})}
\newcommand{\eqqrefs}[1]{Eqs.~(\ref{#1})}

\newcounter{Alg}
\newenvironment{algorithm}[1]{
\refstepcounter{Alg}{\flushleft\bf Algorithm \theAlg.}\ \textit{#1}}
{}

\title{On the Solution to a Countable System of Equations \\ Arising in 
Stochastic Processes}
\author{M.N. Katehakis, L.C. Smit,  F.M. Spieksma}

\begin{document}

\maketitle
\section*{Abstract}
In this paper we 
 develop a  method to  compute the solution to 
 a countable (finite or infinite) 
 set of equations that occurs in many different fields including Markov processes
 that model  queueing systems, birth-and-death processes and inventory systems.
 
 The method provides a fast and exact computation of the inverse of the matrix  of the coefficients 
 of the system.  In contrast, alternative inverse techniques perform much slower and work only for finite size matrices.

Furthermore, we  provide a procedure to construct the eigenvalues of the matrix under consideration.


\section{Introduction}
\subsection{Motivation}
In this paper we 
 develop a  method to  compute the solution to 
 a countable (finite or infinite) 
 set of equations that occurs in many different fields and systems including Markov processes
 that model  queueing systems, birth-and-death processes and inventory systems. For such systems 
in order  to compute performance measures and other quantities of interest, it is often required 
to  invert this matrix.   A class of problems for  which the inverse of this matrix must be computed is given in \cite{Katehakis2015comparative}.

 The method provides a fast and exact computation of the inverse of the matrix  of the coefficients 
 of the system.  In contrast, alternative inverse techniques perform much slower and work only for finite size matrices.  The  more relevant alternative methods  are discussed in this paper,
  for comparison purposes. It is shown that 
  although some cover more general classes of matrices, the method developed in this paper provides a procedure for matrices of infinite size and outperforms all alternative methods in speed.
       As far as we could find,  there are no results   in literature  for countable sized matrices of this form. Existing algorithms for finite matrices perform in general slower than the method we provide.

Apart from the inverse, we have also identified a fast way to compute the eigenvalues of the matrix under consideration, starting with those of a much easier to analyze matrix, one with a birth-and-death structure. Knowing the values of these eigenvalues or just their bounds aids enormously in analysing the specifics of the models of interest.

This paper is organized as follows. First we will introduce some of the notation used in this paper and necessary to understand the review of other methods. Second in Section~\ref{sec:lit}, we will identify some of the existing procedures and specify in what directions they overlap the method discussed in this paper. In Section~\ref{sec:inv} we formally provide the method with its specifications subdivided in four possible matrix forms.  Specifically, see Algorithm~\ref{Alg:Invd}, for the computation of the inverse.
Next, in Section~\ref{sec:eigen} we will exploit the structure of the matrix and derive some results regarding the eigenvalues of the specific matrix.
Finally, in Section~\ref{sec:applications} we will give several applications wherein this matrix-structure appears naturally. Further we obtain  Algorithm~\ref{Alg:HomInv} and \ref{Alg:InvBD} for the cases of element homogenous $B$ and birth-and-death processes with an absorbing state.

\subsection{Preliminaries}
In this paper we develop an efficient computation procedure for the solution vector   $x$  of size 
 $\ell+1$  ($\ell \le \infty $) 
of  a possibly countable system of linear equations of the form 
\begin{equation}\label{eq:bx}
xB=y,
\end{equation}
where   $y$  is a vector  of size 
 $\ell+1   ,$  $B$ is a matrix of size 
$(\ell+1)\times(\ell+1)$ and  it has a structure of the form below:

%
%
\begin{equation}\label{eq:b}
B= \left[\begin{array}{rccccc}
-b^d_0-b^u_0&b^u_0&0&0&0&\cdots\\ 
b^d_1+b^z_1&-b^w_1&b^u_1&0&0&\cdots\\ 
b^z_2&b^d_2&-b^w_2&b^u_2&0&\ddots\\
b^z_3&0&b^d_3&-b^w_3&b^u_3&\ddots\\
b^z_4&0&0& b^d_4&-b^w_4&\ddots\\
\vdots&\vdots&\ddots&\ddots&\ddots&\ddots
\end{array}\right] ,
\end{equation}

where in the above $b^j_i$ ($i=1,2,3,\ldots$ and $j=d,w,u$) are non-negative real numbers and they 
must satisfy the conditions below:
\begin{align}
b_{i}^{z}+b_{i}^{d}+b_{i}^{u} & =b_{i}^{w} ,\label{eq:rowsum}\\
b_{i}^{w} & >0 ,\label{eq:posum}\\
  b^d_{0} & > 0 . \label{eq:bod}
\end{align}

To compute the solution of   \eqqref{eq:bx}  one needs to compute the inverse of matrix $B,$ i.e., the 
matrix $C=B^{-1}$. In Section \ref{sec:inv} we provide constructive algorithms of computing $C$. These methods establish the existence of the inverse and aid to derive a way to truncate a matrix $B$ with arbitrary precision when it is of infinite size. 
   
In this paper we will distinguish $4$ cases: we consider both a finite and an infinite sized matrix, and in both cases the transitions can be of a homogenous or non-homogenous structure. All of these four configurations follow a similar procedure, but have some technical differences that can enhance the speed and usability of the algorithm because of their specific structure.

\subsection{Related Literature}\label{sec:lit}
The computation of the inverse of a general matrix is a procedure with a relatively high complexity when the size of the matrix is large. Therefore, there is an extensive literature devoted to 
  studying this problem for matrices of special structure cf. \cite{ben2003g},  \cite{barrett1981inverses} and \cite{kilicc2013inverse},  \cite{mallik2001inverse}. 

In this section   we  discuss the main existing 
  procedures to compute the inverse of matrix $B$ of the form of \eqqref{eq:b}, 
   we will assume that $\ell$ is finite, since these   procedures  are mainly  applicable for finite sized matrices only.
Further,   
   it is convenient  to change notation and take 
  $B$ to be a matrix of size $\ell\times \ell$. Finally,  we   will decompose $B$ as a combination of the following matrices:
\begin{equation}\label{eq:buw}
B=\widetilde{U}-W,
\end{equation}
where  $$\widetilde{U}=u \delta$$ with $u$ a column vector and $\delta$ a row vector identically equal to zero and with a $1$ as its first entry.

Computing the inverse of $W$ can be done in various ways, since $W$ has such a specific structure.
 Below we describe the main existing approaches in the literature that are based on  using different properties of $W$.

First, we can analyse $W$ from a tridiagonal form perspective (see \cite{mallik2001inverse}). A tridiagonal matrix is a band matrix (see \cite{barrett1981inverses} and \cite{kilicc2013inverse}) with a bandwidth of $2$; only the main diagonal and the two adjacent diagonals contain nonzero elements.
In \cite{mallik2001inverse} a procedure to construct a LU decomposition is provided for tridiagonal matrices.
However to multiply a lower diagonal matrix with an upper diagonal matrix still takes $\mathcal{O}(N^{3})$ elementary operations. From \cite{heinig1984algebraic}, page 61, we construct the inverse of $B$ using the inverse of $W$, if it exists:
\begin{equation}\label{eq:BinvW}
B^{-1}=W^{-1}+W^{-1}ua^{-1}\delta W^{-1},
\end{equation}
where the scalar $a$ is defined as:
$$a=1-\delta W^{-1} u.$$
With a similar argument that has been used to prove that $B$ is invertible, we can prove that the determinant of $W$ is nonzero.
The matrix multiplications that are used to construct $B^{-1}$ using Eq.~(\ref{eq:BinvW}) from the separate elements are of $\mathcal{O}(\ell^{2})$, as is shown below. We will compute $B^{-1}$ in this lineup:
$$B^{-1}=W^{-1}+a^{-1}W^{-1}u\delta W^{-1}=W^{-1}+a^{-1}((W^{-1}u)\delta) W^{-1}$$
In this order, each of the multiplication steps requires at most $\mathcal{O}(\ell^{2})$ operations. Therefore, if the inversion of $W^{-1}$ takes at most $\mathcal{O}(\ell^{2})$ operations, the inversion of $B$ has the same order of operations, since the other action are vector times matrix multiplications.

Second, we note that if matrix $B$ is element homogenous, the matrix $W$ is (almost) of Toeplitz form (cf.~\cite{ammar1996classical, ammar1988superfast, martinsson2005fast}), only the lower right corner element is inconsistent with this classification, since the matrix is of finite dimension and the row sum of row $\ell$ is zero.
We can make approximations to estimate the influence of this disruption.
Many of the methods associated with Toeplitz matrices  are devoted on  finding  the solution of the system $Wx=k$ where $x$ is unknown and $k$ a vector of size $N$. For the purposes that we have in mind, we will need to compute \emph{every} element explicitly and most algorithms are not designed to do so. However, some of them are superfast (i.e.~of $\mathcal{O}(N\,log(N))$), and are useful in other related settings.

In \cite{li2010inverses}, Section 3.3, a fast algorithm that can be applied to matrix $W$ (an M-matrix in that paper) is provided for inverting a tridiagonal matrix in an efficient way. However, the procedure does not readily extend to countable matrices. Our method can be extended, since it uses the extra property that the row sum is zero for all rows except the first. Furthermore, the algorithms of \cite{li2010inverses} require that $b^{d}_{i}b^{u}_{i}\neq 0$ for all $i\in\{2,3,\ldots,N\},$ while the algorithms in this paper do not require these properties.
The proposed algorithm of \cite{li2010inverses} requires to consider the all elements of the matrix, to express off diagonal elements in terms of a diagonal element. The diagonal elements are constructed by a recursive formula that uses the last element of this sequence to explicitly express the first elements. In the non-homogenous case, we encounter the same challenges for infinite state spaces as other algorithms. However, with our procedure we get a sequence that converges to zero, in contrast to the ones in \cite{li2010inverses}. This makes estimations slightly easier to manage. For algorithms regarding the homogenous structure we do not have this issue.

Also, the algorithms in the current paper readily extend to cases where the first two columns contain nonzero elements, possibly even non homogenous and not of the $c\cdot r$ form, something that the algorithm in \cite{li2010inverses} combined with Equation~(\ref{eq:BinvW}) can not do.

Besides their own algorithm, the article of Li et al.~(\cite{li2010inverses}) contains a clarifying table (cf.~Table 1, page 979 of that paper) with several different algorithms of other papers and their corresponding algorithms. It confirms that various other algorithms exist that have the same speed as the ones presented in this paper, but are not easily extendable to countable cases by their nature.

Fourth, matrix $W$ satisfies the framework of a Hessenberg matrix, since it is the summation of a lower diagonal matrix and a tridiagonal matrix.
In \cite{ikebe1979inverses} an algorithm is provided to find two vectors $x$ and $y$ such that $xy$ is the upper part of the inverse, i.e.~$c_{ij}=x_{i}y_{j},$ and similarly for the lower part of the algorithm. However, this is not a computationally efficient algorithm, since it is designed to consider a much larger class of matrices.

To conclude this comparison, our method works for non homogenous cases, can be extended to matrices with two (or more) nonzero columns that are not in the same span, is correct for infinite state spaces and gives an explicit solution independent of the truncation size.
Therefore we think that this algorithm is a major contribution to existing approaches.

\section{Efficient Computation of the Inverse  of  Matrix \emph{B}}\label{sec:inv}
In this section we will describe the four appearances of matrix $B$ and describe their solution procedures. The general idea of the methods is similar, but their are specific differences and simplifications, specific for each case.
 \subsection{The Non Homogenous - Infinite Dimensional Case}
In the sequel, for notational simplicity we let $C$ denote the inverse of $B$, \ie $C:=B^{-1}$. The $(i,j)th$ element of $C$ will be denoted respectively by $c(i,j).$  We will use the notation $B_i$ and $B'_j$ (respectively $C_i$ and $C'_j$) to denote the $i^{th}$ row and 
$j^{th} $ column of matrix $B$ (respectively matrix $C$);  $i,j=0,1,\ldots,\ell$.
   Algorithm \ref{Alg:Invd} below   is based  on the results of  Propositions~\ref{prop:firstcol}-\ref{prop:ldiagd}. 
It  finds  a  matrix $C$ that satisfies  $CB=BC=I$, by computing  recursively the so far unknown elements of the sequence $C'_0,$  $C_0,$ $C'_1, C_1, C'_2,C_2,\ldots, C'_{n-1},C_{n-1},C'_n$   
  for increasing $n \ge 1,$ as described below. 
The algorithm depends on the computation of a sequence of constants, $\gamma_i$, $i=1,2,\ldots, $
that can be computed if we assume that $b^{d}_{i}>0$ for all $i=0,1,\ldots.$ In addition we first require that $b^{u}_{i}>0$: for simplicity we present the algorithm under this assumption. In Proposition \ref{prop:gamma1}  we will show how   $\gamma_i $, $i=1,2,\ldots, $ can be computed when  $b_{i}^{d}=0  $ for some $i.$ In addition, we will briefly explain how adjustments can be made to handle $b^{u}_{i}=0$ for some $i$.

\begin{algorithm}{  }\label{Alg:Invd}
\begin{itemize}
\item[]
At stage 0: 
\begin{enumerate}[label=\alph*)]
\item
The $0^{th}$  column of $C$ (i.e., the column containing elements $c(i,0)$) is computed  using  Eq.~(\ref{eq:c1})  of Proposition \ref{prop:firstcol}. 
\item   All elements of the $0^{th}$ 
 row of  $C$ (elements $c(0,j)$) are computed  using  Eqs.~(\ref{eq:defgamma}) and (\ref{eq:gam})  of Proposition \ref{prop:rdiag},
where $\gamma_{1}$ is given by  Eq.~(\ref{eq:gammai})  of Proposition  \ref{prop:udiag}.
\end{enumerate}
\item[]
At stage $i= 1,2,\ldots,\ell$\ : 
\begin{enumerate}[label=\alph*)]
\item
For the $i^{th}$   column  $C'_i$, since the $c(0,i),\ldots,c(i-1,i)$ have already been computed,
we calculate  the remaining  by Eq. (\ref{eq:ci}) of Proposition \ref{prop:ldiag}. 
\item
For the $i^{th}$   row  $C_i$, since the $c(i,0),\ldots,c(i,i)$ have already been computed,
we calculate  the remaining  by Eq. (\ref{eq:cid}) of Proposition \ref{prop:ldiagd}. 
\end{enumerate}
\end{itemize}
\end{algorithm}

In the following propositions we will show that the algorithm above is correct.

\begin{proposition}\label{prop:firstcol}
The following is true for all $i\geq 0:$ 
\begin{equation}
\label{eq:c1} c(i,0) =-1/b_{0}^{d}.
\end{equation}
\end{proposition}
\begin{proof}
By considering the set of equations  $BC'_0=\delta'_0$ we find:
\begin{align*}
(-b^d_{0}-b^u_{0})c(0,0) + b^u_{0} c(1,0)&= 1,   \\
 b^z_{i} c(0,0) + b^d_{i} c(i-1,0) -b^w_{i} c(i,0) + b^u_{i} c(i+1,0)&= 0,\  \mbox{for} \ i \geq 1.
\end{align*}
Using the definition of $b^{w}_{i}$ in Eq.~(\ref{eq:rowsum}) it is easy to see that for all $i\geq 0$,
$c(i,0) =-1/b_{0}^d,$  is a solution that suffices and therefore has to be the unique solution.
\end{proof}

Proposition \ref{prop:rdiag} below shows that the elements $c(0,j)$ 
 with $j>0$ depend only on $c(0,0)$. Note that $c(0,0)$ is nonzero by its construction in Proposition~\ref{prop:firstcol}.
 
\begin{proposition}\label{prop:rdiag}Define  
\begin{equation}\label{eq:defgamma}
\gamma_{j}:=c(0,j)/c(0,0).
\end{equation} 
There exist scalars $\rho_{j}$, $\eta_{j},$ with $j=1,2,\ldots$ such that all constants $\gamma_{j}$ can be recursively computed as a function of $\gamma_{1}$  as follows:
  \begin{equation}\label{eq:gam}
  \gamma_{j}=\rho_{j}\gamma_{1}+\eta_{j},
  \end{equation}
where the constants  $\rho_{j}$ and $\eta_{j}$ are given by, under the assumption that $b_{j}^{d}>0$:  
  \begin{equation}\label{eq:rho}\rho_{j}=\frac{b_{j-1}^{w}\rho_{j-1}-b_{j-2}^{u}\rho_{j-2}}{b_{j}^{d}} 
   \end{equation}
 and 
  \begin{equation}\label{eq:eta}\eta_{j}=\frac{b_{j-1}^{w}\eta_{j-1}-b_{j-2}^{u}\eta_{j-2}}{b_{j}^{d}},
    \end{equation}
 with initial values: $\rho_{1}=1,\,\eta_{1}=0,\,\rho_{2}=b^{w}_{1}/b^{d}_{2},\,\eta_{2}=-b^{u}_{0}/b^{d}_{2}$.
\end{proposition}
\begin{proof}
Equation~(\ref{eq:gam}) follows by  the  systems of  equations $C_0 B=\delta$, without considering the first equality. 
It is easy to see that every equation has the form below, where $j>0$:
\begin{equation}\label{eq:s-gamma}
b_{j-1}^{u} c(0,j-1) -b_{j}^{w}  c(0,j) + b_{j+1}^{d} c(0,j+1)= 0.
\end{equation}

The combination of this recursive structure and the ergodicity assumption (that ensures that $\lim_{j\rightarrow\infty} c(0,j)=0$ for all $j$), allows us to express all elements $c(0,j)$  in terms of their preceding elements. Thus by substitution, we conclude that every element is a product of $c(0,0)$ and a constant, depending on  $j$ that we denote as $\gamma_{j}.$

Eq.~(\ref{eq:gam}) uses Eq.~(\ref{eq:defgamma}) and follows from the observation that:
$$
\gamma_{2}=\frac{b^{w}_{1}\gamma_{1}-b_{0}^{u}}{b^{d}_{2}},
$$
and in general for $j\geq 3$:
\begin{equation}\label{eq:GamRecur}
\gamma_{j}=\frac{b^{w}_{j-1}\gamma_{j-1}-b_{j-2}^{u}\gamma_{j-2}}{b^{d}_{j}}.
\end{equation}

It is clear that $\gamma_{2}$ is of the form described in the Proposition (since $\rho_{2}=b_{1}^{w}/b_{2}^{d}$ and $\eta_{2}=b_{0}^{u}/b_{1}^{d}$) and by substituting, $\gamma_{3}$ has this structure as well. An induction argument completes  the proof.
\end{proof}

Proposition \ref{prop:udiag} provides a method to calculate $\gamma_{1},$ using the expressions derived above.
\begin{proposition}\label{prop:udiag}
The scalar $\gamma_{1}$ can be calculated algebraically as follows:
\begin{equation}\label{eq:gammai} 
\gamma_{1} 
=\frac{b_{0}^{u}-\sum_{j=1}^{\infty}b_{j}^{z}\eta_{j}}
{b^{d}_{1}+\sum_{j=1}^{\infty}b_{j}^{z}\rho_{j}}.\end{equation}
%
\end{proposition}
\begin{proof}
To verify this expression for $\gamma_{1}$, we consider the equation $C_{0}B'_{0}=1.$ The result follows immediately.
\end{proof}


The next proposition provides a method of computing the under diagonal elements ($c(i,j),$ with $j=1,2,\ldots$ and $i=j,j+1,\ldots$) of $C$. 
We denote $\delta_{i,j}$ as a scalar that takes the value $1$ if $i=j$, and $0$ otherwise.
\begin{proposition}\label{prop:ldiag} Assume that  $b^{u}_{i}>0 , $ for all $i=0,1,\ldots .$ 
The following is true for all elements $c(i,j)$ with $i\ge j\ge 1$.
\begin{equation}\label{eq:ci}
c(i,j) =
\begin{dcases}
 -\gamma_{1}(1/b^{d}_{0} + 1/b^{u}_{0}),&\text{for }i=j=1,\\
\frac{1}{b^{u}_{i-1}}(-b^{z}_{i-1} c(0,j) - b^{d}_{i-1} c(i-2,j) +b^{w}_{i-1} c(i-1,j) + \delta_{i-1,j}),&\text{otherwise.}
\end{dcases}
\end{equation}
\end{proposition}
\begin{proof}
To compute $c(1,1)$ we use the   equation $B_0 C'_1 =0,$ \ie
$$
-(b^{u}_{0} +b^{d}_{0}) c(0,1) +b^{u}_{0} c(1,1) = 0,
$$
and the statement is complete since by Proposition~\ref{prop:rdiag} we 
have  $c(0,1) =\gamma_{1} c(0,0)=-\gamma_{1}/b_{0}^{d}$.

%
%

To compute the remaining elements $c(i,j)$ of $C$, we use  $B_{i-1}C'_{j}=\delta_{i-1,j}$. Indeed,   the product of the $(i-1)^{th}$ row of $B$ (where $i\geq 2)$ and the  $j^{th}$ column  of  $C$ (where $j\geq 1$) is the lefthand side of the equation below and the proof is complete.
$$
b_{i-1}^{z} c(0,j) + b_{i-1}^{d} c(i-2,j) -b_{i-1}^{w} c(i-1,j)+b_{i-1}^{u}c(i,j)=\delta_{i-1,j}.$$
\end{proof}

 \begin{proposition}\label{prop:ldiagd} The following is true for $c(i,j)$ with $j> i \ge 1$ and $b_{j}^{d}>0.$
 \begin{equation}\label{eq:cid}
 c(i,j)=\frac{b_{j-1}^{w}  c(i,j-1)-b_{j-2}^{u} c(i,j-2)+\delta_{i+1,j}}{b_{j}^{d}}.
\end{equation}
\end{proposition}
\begin{proof}
Equation~(\ref{eq:cid}) follows from  the  systems of  equations $C_i B=\delta_{i}$, without considering the first $i-1$ equalities. 
Every equation has the form below, where $j>0$:
\begin{equation*}
b_{j-2}^{u} c(i,j-2) -b_{j-1}^{w}  c(i,j-1) + b_{j}^{d} c(i,j)= \delta_{i+1,j}.
\end{equation*}
\end{proof}

\subsubsection{Zeros above and below the diagonal}
In this section we discuss how Algorithm~\ref{Alg:Invd} can be extended to allow $b^{d}_{i}=0$ and $b^{u}_{i}=0$ to be zero. The elements $b^{z}_{i}=0$ can always be zero and do not influence the procedure of the algorithm.
When $b_i^d=0$  for some $i$,
   Eq.~(\ref{eq:GamRecur}) can not be used  to compute the corresponding  $\gamma_i$  because of the devision by $0$ that occurs. 
However, we next show that  under the conditions in \eqqrefs{eq:rowsum}-(\ref{eq:bod}) the  matrix $C$ still exists and this inverse can be computed by 
Algorithm \ref{Alg:Invd} with  a small modification to compute    $\gamma_{i} $  in the following way.

Let  $I_0=\{i_k,$ $k=1,\ldots,\nu, \  \text{such that} \ b_{i_k}^{d}=0 \}$ and let $i_0=1$ and $i_{\nu+1}=\infty.$

\begin{proposition}\label{prop:gamma1}
Assume that $I_{0}$ is not empty, then for $k=0,1,2,\ldots,\nu-1$: 
\begin{equation}\label{eq:gamma1}
\gamma_{i_k}=
\begin{dcases}
\frac{b^{w}_{i_{k+1}-1}\eta_{i_{k+1}-1}-b^{u}_{i_{k+1}-2}\eta_{i_{k+1}-2}}{-b^{w}_{i_{k+1}-1}\rho_{i_{k+1}-1}+b^{u}_{i_{k+1}-2}\rho_{i_{k+1}-2}},&\text{if } i_{k+1}\neq i_{k}+1,\\
\frac{b^{u}_{i_{k}-1}\gamma_{i_{k}-1}}{b^{w}_{i_{k}}},& \text{if } i_{k+1}=i_{k}+1.
\end{dcases}
\end{equation}
 
For $i_\nu$:
\begin{equation}\label{eq:gammanu}
\gamma_{i_\nu}=\frac{b_{0}^{u}-\sum_{j=1}^{i_\nu-1}b_{j}^{z}\gamma_{i}-b^{d}_{1}\gamma_{1}-\sum_{j=i_\nu}^{\infty}b_{j}^{z}\eta_{j}}
{\sum_{j=i_\nu}^{\infty}b_{j}^{z}\rho_{j}}.
\end{equation}
The remaining components of $\gamma$ are constructed for $j\in \{i_{k}+1,\ldots,i_{k+1}-1\}$ as: 
\begin{equation*}
\gamma_{j}=\rho_{j}\gamma_{i_{k}}+\eta_{j}.
\end{equation*}
where the vectors $\rho$ and $\eta$ are as before in \eqqref{eq:rho} and \eqqref{eq:eta}, specified below, with $j=1,2,\ldots.$
$$
\rho_{j}=
\begin{dcases}
1,&\text{if }b^{d}_{j}=0 \text{ or if } j=1,\\
\frac{b^{w}_{j-1}}{b^{d}_{j}},&\text{if } b^{d}_{j}\neq 0 \text{ and } b^{d}_{j-1}=0,\text{ or if } j=2,\\
\frac{b_{j-1}^{w}\rho_{j-1}-b_{j-2}^{u}\rho_{j-2}}{b_{j}^{d}},
&\text{if } b^{d}_{j}\neq 0 \text{ and } b^{d}_{j-1}\neq 0,\ j\ge 3,\\
\end{dcases}
$$
and:
$$
\eta_{j}=
\begin{dcases}
0,&\text{if }b^{d}_{j}=0 \text{ or if } j=1,\\
\frac{-b^{u}_{j-2}\gamma_{j-2}}{b^{d}_{j}},&\text{if } b^{d}_{j}\neq 0 \text{ and } b^{d}_{j-1}=0, \text{ or if } j=2,\\
\frac{b_{j-1}^{w}\eta_{j-1}-b_{j-2}^{u}\eta_{j-2}}{b_{j}^{d}},&\text{if } b^{d}_{j}\neq 0 \text{ and } b^{d}_{j-1}\neq 0,\ j\ge 3,\\
\end{dcases}
$$
where $\gamma_{0}=1.$
\end{proposition}

\begin{proof}
We show how to compute 
$\gamma_1, \gamma_2, \ldots \gamma_{i_1}$. Consider the   steady state equations given in \eqqref{eq:s-gamma}: 
only the equation corresponding to $i=i_1-1$ changes to
\begin{equation*}
 b^{u}_{i_1-2}\gamma_{i_1-2}=b^{w}_{i_1-1}\gamma_{i_1-1}.
\end{equation*}
 Since there is no change in the formula for $i<i_1-1$, \eqqrefs{eq:gam} hold for  $i<i_1-1$, with $\rho $ and $\eta$ defined as above, from which we obtain the equations below (for $i=i_1-1$ and $i=i_1-2$ respectively): 
\begin{equation*}
\gamma_{i_1-1}=\rho_{i_1-1}\gamma_{1}+\eta_{i_1-1}
\end{equation*}
and similarly for $\gamma_{i_1-2}$:
\begin{equation*}
\gamma_{i_1-2}=\rho_{i_1-2}\gamma_{1}+\eta_{i_1-2}.
\end{equation*}
Combining the three equations above, we find the  expression of \eqqref{eq:gamma1} for $\gamma_{1}.$

\eqqref{eq:gammanu} follows as before from $C_0B'_0=1 ,$ using the computed $\gamma_{i}$ by \eqqref{eq:gamma1}, $i=1,\ldots,i_{\nu}-1$.

To find the subsequent values of $\gamma_{i}$ we  repeat this  process starting at the beginning of the 
next segment of the set $\{ \gamma_{i_1+1},\ldots, \gamma_{i_2} \} .$ The proof is complete noting that 
the case $\nu=\infty$ corresponds to  an infinite set 
 $\{ \gamma_{i_\nu-1},\ldots  \}$ on which the approach of the algorithm is applied. 
\end{proof}

In the sequel when we refer to Algorithm \ref{Alg:Invd} we include the possible modification applied per Proposition  \ref{prop:gamma1}.

What adjustments need to be made to compute the remaining elements when a division by zero is required for the algorithm (i.e.~$b^{u}_{i-1}=0$ in Eq.~(\ref{eq:ci})), are very similar to the procedure described in Proposition~\ref{prop:gamma1}. We `shift' away from this equation and use the next equations, to express following elements in terms of the one that was not computable because of the undefined division. This procedure happens horizontally for upper diagonal elements when $b^{d}_{i}=0$, and vertically for under diagonal elements when $b^{u}_{j}=0.$
 
%
%
%
%
%
%

Note that when $b^{u}_{i}=0$ then $c(k,l)=0$ for $1\le k \le i$ and $l \ge i+1.$
\subsubsection{Multiple nonzero columns}
One way to enlarge the class of matrices for which the method above can be applied is by allowing other columns to be nonzero for every element. Matrix $B$ will now have the following form:

\begin{equation}\label{eq:b2col}
B= \left[\begin{array}{rccccc}
-b^d_0-b^u_0&b^u_0+b^{z_{2}}_{0}&0&0&0&\cdots\\ 
b^d_1+b^{z_{1}}_1&-b^w_1+b^{z_{2}}_{1}&b^u_1&0&0&\cdots\\ 
b^{z_{1}}_2&b^d_2+b^{z_{2}}_{2}&-b^w_2&b^u_2&0&\ddots\\
b^{z_{1}}_3&b^{z_{2}}_{3}&b^d_3&-b^w_3&b^u_3&\ddots\\
b^{z_{1}}_4&b^{z_{2}}_{4}&0& b^d_4&-b^w_4&\ddots\\
\vdots&\vdots&\ddots&\ddots&\ddots&\ddots
\end{array}\right].
\end{equation}
As before, the row sum is zero for every row, i.e. for all $i$:
$$b_{i}^{z_{1}}+b_{i}^{z_{2}}+b_{i}^{d}+b_{i}^{u} =b_{i}^{w}.$$

Algorithm \ref{Alg:Invd} can still be applied, the only adjustments that need to be made are briefly described below. We will leave the details to the reader.
The main difference with the original algorithm is that in this case, we express all elements of the top row in terms of the top row element corresponding to the nonzero column with the highest index.

As long as the number of nonzero remains small, the order of the algorithm remains the same, only the normalization takes some extra steps. When the number of nonzero columns becomes countable, the algorithm loses its computational advantage and other methods might be preferable.

Another approach is to use the Toeplitz equation, cf.~\eqqref{eq:BinvW}, when the second column is a multiple of the first column. This is a necessary condition to use this equation.
 
 \subsection{The Non Homogenous - Finite Dimension Case}
The case in which  $\ell+1 <\infty$, may happen due to modeling or  truncation. 
 No adjustments are needed in the formulation of Algorithm \ref{Alg:Invd}, since the only difference is the altered condition at the boundary:
 $$ b^w_\ell=b^d_\ell+b^z_\ell ,$$
i.e.,  the row sum of the final row is zero. The formulas   of Propositions 2 - \ref{prop:gamma1} are adjusted so that the corresponding summations are up to $\ell $ where $\ell$ is finite.

 \subsection{The   Homogenous - Infinite Dimension Case}

In this paragraph we define the following subclass of matrices $B$ that possess  the structure described in Eq.~(\ref{eq:b}) and 
the  additional structure of the definition below.
\begin{definition}\label{def:eh}
$B$ is called  {\sl element homogeneous} when the following relations hold
$$b^{d}_{i}=b^{d}, b^{u}_{i}=b^{u}, b^{z}_{i}=b^{z} \mbox{\ for all $i$.} $$
\end{definition}
 
We note that when   $B$  is {\sl element homogeneous}  then necessarily the following also holds: $b^w_{i}=b^w$, for all $i.$

In the rest of this section we denote the constant $\gamma_{1}$ as $\gamma$ and derive the following result. We assume that none of the elements given are zero to avoid trivialities and more tedious notation. First we derive the following result regarding the scalar $\gamma$.
 
\begin{proposition}\label{prop:hom}
When $B$ has a homogenous structure and the size $\ell$ is countable, the following is true for all $j\ge 1$:
\begin{equation}\label{eq:gammahi}\gamma_{j}=\gamma^{j},
\end{equation}
where
\begin{equation}\label{eq:gammah}
\gamma=\frac{b^{w}-\sqrt{(b^{w})^{2}-4b^{u}b^{d}
}}{2b^{d}}.
\end{equation}
\end{proposition}
\begin{proof}
The proof follows by considering the  systems of  equations $C_i B'_{j}=0$ for $j=i+1,i+2,\ldots$. 
All equalities have the form below.\begin{equation*}
b^{u} c(i,j-1) -b^{w}  c(i,j) + b^{d} c(i,j+1)= 0.
\end{equation*}
Because of the homogenous structure, it is clear that the corresponding elements of the row $C_{0}$ have a product form with a constant  factor $\gamma$, e.g. for $j=i:$
$$
b^{u} c(i,j) -b^{w}  \gamma c(i,j) + b^{d} \gamma^{2}c(i,j)= 0,
$$
thus in other words:
\begin{equation}\label{eq:gamchom}
c(i,j)=\gamma^{j}c(i,i).
\end{equation}
When, solving  for $\gamma$ the above simplifies to 
$b^{d} \gamma^2-b^{w}\gamma+b^{u}=0$. This quadratic equation has  
 a unique solution between $0$ and $1$ given by \eqqref{eq:gammah}. This solution is greater than $0$ since $\sqrt{(b^{w})^{2}-4b^{u}b^{d}}<\sqrt{(b^{w})^{2}}=b^{w}$. This $\gamma$ is smaller than $1$ since:
 \begin{align*}
\gamma=\frac{b^{w}-\sqrt{(b^{w})^{2}-4b^{u}b^{d}}}{2b^{d}}&<\frac{b^{w}\sqrt{(b^{w})^{2}-4b^{u}b^{d}-4b^{z}b^{d}}}{2b^{d}}\\
&=\frac{b^{w}-\sqrt{(b^{w})^{2}+4(b^{d})^{2}-4b^{w}b^{d}}}{2b^{d}}\\
&=\frac{b^{w}-\sqrt{(b^{w}-2b^{d})^{2}}}{2b^{d}}\\
&=\frac{2b^{d}}{2b^{d}}=1.
\end{align*}
\end{proof}

When $B$ is element homogenous, we do not need Proposition \ref{prop:udiag} to compute $C,$ since $\gamma_{1}=\gamma$ and computable by Proposition \ref{prop:hom}. Proposition \ref{prop:ldiag} remains the same when $B$ is element homogenous.
In addition, the expression for the computation of the under diagonal elements can also be simplified. Using both expressions, we can even directly express the diagonal elements in terms of its predecessor.

The above leads to the following considerable  simplification of Algorithm~\ref{Alg:Invd}. 

\begin{algorithm}{Computation of  $C:=B^{-1}$ for an element homogeneous, countable, matrix $B.$}\label{Alg:HomInv}
\begin{itemize}
\item[]
At stage 1: 
\begin{enumerate}[label=\alph*)]
\item
Calculate $\gamma$ using Eq. (\ref{eq:gammah}).
\item
Calculate $\psi$ using \eqqref{eq:psi}.
\item
The remaining elements of the $0^{th}$  row of  $C$ are computed  using   Proposition \ref{prop:hom} and in particular \eqqref{eq:gamchom} in that proposition.
\end{enumerate}
\item[]
At stage $i= 1,2,\ldots,\ell$\ : 
\begin{enumerate}[label=\alph*)]
\item The diagonal element $c(i,i)$ is computed by \eqqref{eq:diagonal}.
\item The elements of $C_{i}$, the $i^{th}$ row of $C$ to the right of the diagonal  are computed using Proposition \ref{prop:hom} and in particular \eqqref{eq:gamchom} in that proposition.
\item The under diagonal elements of $C'_{i},$ the $i^{th}$ column of $C$ are computed using
Eq. (\ref{eq:cijunder}) of Proposition~\ref{prop:cijunder}.
\end{enumerate}
\end{itemize}
\end{algorithm}

To summarize, the elements $c(i,j)$ of $C$ can be calculated as follows:
\begin{equation}\label{eq:cijhom}
c(i,j) =\begin{dcases}
\frac{-1}{b^d }& \text {if} \ j= i=0, \\
\frac{-1+b_{u}((1-\psi) c(0,i-1) + \psi c(i-1,i-1))}{b_{w}-b_{d}\gamma}
 & \text {if} \ j = i\ge 1,\\
\gamma^{j-i} c(i,i) & \text {if} \ j\ge i+1, \\
(c(j,j)-c(0,j))\psi^{i-j}+c(0,j) & \text {if} \ j \le i-1.
\end{dcases}
\end{equation}
For a feasible way to compute all elements, we refer to the order in Algorithm~\ref{Alg:HomInv}.

Below we will prove that Algorithm \ref{Alg:HomInv} is correct. As part of this algorithm, the under diagonal elements of column $C'_{i}$ are expressed in terms of its first element $c(i,0)$ and its diagonal element $c(i,i)$. This calculation is given in the proposition below.

\begin{proposition}\label{prop:cijunder}
The under diagonal elements of $C$  can be calculated as follows:
\begin{equation}\label{eq:cijunder}
c(i,j)=(c(j,j)-c(0,j))\psi^{i-j}+c(0,j),
\end{equation}
where
\begin{equation}\label{eq:psi}
\psi=\frac{b^{w}-\sqrt{(b^{w})^{2}-4b^{u}b^{d}}}{2b^{u}}=\gamma b^{d}/b^{u},
\end{equation}
and where $\psi$ is between $0$ and $1$.
\end{proposition}
\begin{proof}
We consider the equation $B'_{i}C_{j}=0,$ with $i\ge j+1$:
$$
b^{z}c(0,j)+b^{d}c(i-1,j)-b^{w}c(i,j)+b^{u}c(i+1,j)=0.
$$
Secondly, we define $d(i,j):=c(i,j)-c(0,j)$ for all $i\ge j.$ Then it is easy to see that the above is the same as:
$$
b^{d}d(i-1,j)-b^{w}d(i,j)+b^{u}d(i+1,j)=0.
$$
Like in the proof of Proposition \ref{prop:hom} these elements $d(i,j)$ have a product form in $i$ with a constant  factor $\psi$, for all $i\ge j+1$:
$$
b^{d} d(i-1,j) -b^{w}  \psi d(i,j) + b^{u} \psi^{2}d(i+1,j)= 0,
$$
thus in other words:
$$
d(i,j)=\psi^{i-j}d(j,j).
$$
When solving  for $\psi$, this quadratic equation has  
 a unique solution between $0$ and $1$, given by \eqqref{eq:psi}. 
Next, substituting $d(i,j)=c(i,j)-c(0,j)$ we get:
$$
d(i,j)=c(i,j)-c(0,j)=\psi^{i-j}(c(j,j)-c(0,j)).
$$
and the proof is complete.
\end{proof}

Since $\psi$ and $\lambda$ are expressed explicitly, we can calculate the diagonal elements using Equation~\eqqref{eq:diagonal} below. This is true by isolating $c(i,i)$ in $C_{i}B'_{i}=1.$

\begin{equation}\label{eq:diagonal}
c(i,i)=(-1+b_{u}((1-\psi) c(0,i-1) + \psi c(i-1,i-1)))/(b_{w}-b_{d}.\gamma) 
\end{equation}

In the next proposition we derive an interesting result regarding the convergence of the diagonal elements, where we abbreviate $D:=b^{2}_{w}-4b_{u}b_{d}$. Note that $D>0.$

\begin{proposition}
The diagonal elements of C converge to:
$$
\lim_{i\rightarrow\infty} c(i,i)=\frac{-1}{\sqrt{D}}.
$$
\end{proposition}
\begin{proof}
First, it is important to note that by definition, $b^{d}\gamma^{2}-b^{w}\gamma+b^{u}=0$, and thus:
$$b^{d}\gamma^{2}=b^{w}\gamma-b^{u}.$$

We want to show that the diagonal elements $c(i,i)$ converge to a (negative) constant, for all $i=1,2,\ldots$.
 Therefore we define the dif{}ference $v(i)$ between two subsequent diagonal elements as:   
 $$v(i):=c(i,i)-c(i-1,i-1).$$ 
 In addition, we use an alternative expression for $c(i,i)$ in the equality $B_{i-1}C'_{i}=0$, given in \eqqref{eq:ab-diag} below.
\begin{equation}\label{eq:ab-diag}
c(i,i)=(-b_{z} c(0,i) -b_{d} \gamma^{2} c(i-2,i-2) + b_{w}\gamma c(i-1,i-1))/b_{u}.
\end{equation}

We find a recursive expression for $v(i)$, using \eqqref{eq:ab-diag} and the following steps:
\begin{align*}
-b_{z}c(0,i)=\gamma^{i}b_{z}/b_{d}&=b_{u} c(i,i)-b_{w}\gamma c(i-1,i-1) + b_{d}\gamma^{2}c(i-2,i-2)\\
&=b_{u} \left(c(i,i)-c(i-2,i-2)\right)-b_{w}\gamma \left(c(i-1,i-1)-c(i-2,i-2)\right)\\
&=b_{u}\left(v(i)-v(i-1)\right)-b_{w}\gamma v(i-1),\\
\end{align*}
and thus:
\begin{equation}
v(i)=\gamma^{i}b_{z}/b_{d}b_{u}+\frac{b_{w}\gamma-b_{u}}{b_{u}}v(i-1).
\end{equation}
In the sequel we will use that $\gamma^{i}\rightarrow \infty$ and that:
\begin{align*}
\frac{b_{w}\gamma-b_{u}}{b_{u}}&=b_{d}\gamma^{2}/b_{u},\\
&=\psi\gamma.
\end{align*}
Note that since both $\psi$ and $\gamma$ are  between $0$ and $1$, there product is in the same interval.
We can use all of the above to prove the convergence of $c(i,i)$ by considering the absolute value of the difference function $v(i)$.
\begin{align*}
\lim_{i\rightarrow \infty} |v(i)|&=\lim_{i\rightarrow \infty} | \gamma^{i}b_{z}/(b_{d}b_{u})+\frac{b_{w}\gamma-b_{u}}{b_{u}}v(i-1) |\\
&=\frac{b_{d}\gamma^{2}}{b_{u}}\lim_{i\rightarrow \infty} | \gamma^{i-2}b_{z}/b_{d}^{2}+v(i-1) |\\
&\leq
\psi\gamma\lim_{i\rightarrow \infty} | \gamma^{i-2}b_{z}/b_{d}^{2}|+\lim_{i\rightarrow \infty}|v(i-1) |\\
&=\psi\gamma\lim_{i\rightarrow \infty}|v(i-1)|\\
&=\psi\gamma\lim_{i\rightarrow \infty}|v(i)|,
\end{align*}
and thus zero.
Since we have found that the series converges, we can explicitly express this limit and complete the proof.
\begin{align*}
\lim_{i\rightarrow \infty}c(i,i)&=\lim_{i\rightarrow \infty}(-1+(b_{u}+b_{z}-b_{u}\psi) \gamma^{i} c(0,0) +b_{d} \gamma c(i-1,i-1))/(b_{w}-b_{u}\psi)\\
&=-1/(b_{w}-b_{u}\psi)+b_{d}\gamma/(b_{w}-b_{u}\psi) \lim_{i\rightarrow \infty}c(i,i)\\
&=-1/(b_{w}-b_{d}\gamma-b_{u}\psi)\\
&=\frac{-1}{\sqrt{D}}.
\end{align*}
\end{proof}

 \subsection{The   Homogenous - Finite Dimension Case}
In this section we consider the element homogenous and finite version of matrix $B$. 
To be able to use Algorithm \ref{Alg:HomInv} and thus to compute the row and column independent scalars $\gamma$ and $\psi$, it is necessary that the state space is truncated in a specific way. This truncation is described in the proposition below.
\begin{proposition}
Consider a finite and element homogenous matrix $B$ of size $\ell\times\ell$ satisfying \eqqref{eq:b}. If:
$$B(\ell,\ell)=-b^{u}/\gamma$$
and
$$B(1,\ell)=b^{u}/\gamma-b^{d}$$
then
$c(i,j)$ can be computed per \eqqref{eq:cijhom}
 where $\gamma$ is defined as in \eqqref{eq:gammah} and $\psi$ as in \eqqref{eq:psi}.
\end{proposition}
\begin{proof}
When considering  $C_{i}B'_{\ell}=0$ with $0\le i\le \ell-1$
we find that:
$$
b^{u}c(i,\ell)-b^{u}/\gamma c(i,\ell)=0
$$
and thus that $c(i,\ell)=\gamma c(i,\ell-1)$. All elements $c(i,j)$ can now be expressed in terms of $\gamma$ and its predecessor analogous to Proposition \ref{prop:hom}.

When considering  $B_{\ell}C'_{j}=0$ with $0\le i\le \ell-1$ and the observation that $b^{u}/\gamma=b^{d}/\psi$
we find that:
$$
(b^{d}/\psi-b^{d})c(0,j)+b^{d}c(\ell-1,j)-b^{d}/\psi c(\ell,j)=0.
$$
Rewriting gives that:
$$
b^{d}(c(\ell,j)-c(0,j))=\psi b^{d}(c(\ell-1,j)-c(0,j)).
$$
All elements can now be expressed in terms of $\psi$, its predecessor and $c(0,j),$ analogous to Proposition \ref{prop:cijunder}.
\end{proof}

When an alternative truncation of $B$ is chosen or provided, the elements $\gamma$ and $\psi$ turn out to be row and column dependent  and therefore we refer to the solution procedure of a non element homogenous finite matrix $B$.

It is clear that to computation of each element only requires a linear number of steps. Therefore the complexity of Algorithms~\ref{Alg:Invd} and \ref{Alg:HomInv} is $\mathcal{O}(\ell^{2}).$

\section{The eigenvalues of Matrix \emph{B}}\label{sec:eigen}
In this section we will state properties for the eigenvalues $\{\lambda^{b}_{i}\}_{i=1,\ldots,\ell+1}$ of matrix $B$, valid for the two finite appearances of the matrix described in the previous sections.
 
We will use the structure of matrix $W$, introduced in \eqqref{eq:buw}.
In general we do not require that $w_{i,i-1}w_{i-1,i}>0$. When this property is violated, then most of the analysis below does not hold. For this section we therefore do assume this inequality to hold.
 We will make use of the following well-known lemma, proved in several papers.
 
 \begin{proposition}\label{lem:eigW}
 Consider a finite tridiagonal, diagonal matrix $W$ with entries $w_{ij}$ where the product $w_{i,i+1}w_{i-1,i}>0$ and $w_{ii}<0$ for all $i$ and $w_{i,i-1}+w_{i,i}+w_{i,i+1}=0.$ Its eigenvalues $\{\lambda_{i}\}_{i=1,\ldots,\ell+1}$ are negative, real and distinct.
 \end{proposition}
 \begin{proof}
One of the more complicated proofs of this statements uses Sturm sequences. (cf.~\cite{grassmann2002real}) 
Other proofs involve diagonalization, or follow the line of Meurant (cf.~\cite{meurant1992review}).
 \end{proof}
Using the results in the paper of Alexanderian \cite{alexanderian2013continuous}, Remark 3.4 and Remark 3.6 we know that for a small perturbation in matrix $B$ the eigenvalues lie in a balls of radius $\epsilon$ around the eigenvalues of $W$.
The corresponding set of eigenvectors $\{c_{1},c_{2},\ldots,c_{\ell+1}\}$ of $W$ is therefore a complete set of eigenvectors for $A$ and therefore spans the  space $\mathbb{R}^{\ell+1}.$ The first element of $c_{i}$ can not be zero: the equations of $Ac=\lambda c$ imply iteratively that all entries of $c_{i}$ are zero and thus that $c_{i}$ is a null-vector.


 Below we prove a result regarding the eigenvalues of $B$.
 \begin{theorem}
When matrix $B$ is finite, all its eigenvalues have a negative real part.
\end{theorem}
\begin{proof}
Since $B$ is invertible (cf.~\cite{sslqsf2013}) all eigenvalues are nonzero.
By the Gershgorin circle theorem we know that every eigenvalue of  $B$  lies within at least one of the Gershgorin discs  $D(b_{ii},R_i)$ for all $i \in\{1,\dots,\ell+1\}$ where $R_i = \sum_{j\neq{i}} b_{ij}$ and by definition the diagonal elements $b_{ii}<0.$ Since $R_{i}<-b_{ii}$ every  $D(b_{ii},R_i)$ is non positive.
\end{proof}

One possible way to construct matrix $B$ from matrix $W$ is by adding a linear combination of the eigenvectors (since $W$ has full rank, see Proposition~\ref{lem:eigW}) to the first column:
$$B=W+(n_{1}c_{1}+n_{2}c_{2}+\ldots +n_{\ell+1}c_{\ell+1})\delta,$$
where $\delta$ is a row vector with $\delta(0)=1$ and $\delta(i)=0$ elsewhere; the numbers $n_{i}$ are scalars, possibly zero. 
By reordering and renaming such that $\{n_{1},n_{2},\ldots,n_{M}\}\neq 0$ and 
by norming the eigenvectors $c_{i}$ at size $1/n_{i}$, we get without loss of generality that:
\begin{equation}\label{eq:bweig}
B=W+(c_{1}+c_{2}+\ldots +c_{M})\delta,
\end{equation}
with $c_{i}$ a normed (renamed) eigenvector. 
We prove the following proposition, regarding matrix $A_{1}:=A+
\alpha_{1}c_{1}\delta$, for any matrix $A$ (not necessarily a tridiagonal matrix), where use the same notation for the eigenvalues and eigenvectors of $A$ as those of $W$  and  scalar $\alpha_{1}$ is chosen such that $\alpha_{1}\neq \frac{\lambda_{i}-\lambda_{1}}{c_{1}(0)}$ for all $i$.

\begin{proposition}\label{lem:eigvalue}
Matrix $A_{1}$ has eigenvalues $\{\lambda_{1}+\alpha_{1}c_{1}(0),\lambda_{2},\ldots,\lambda_{\ell+1}\}$ and eigenvectors $\{c_{1},c_{2}+x_{1}c_{1},\ldots, c_{\ell+1}+x_{\ell+1}c_{1}\}$, where:
$$x_{i}=\frac{\alpha_{1}c_{i}(0)}{\lambda_{i}-\lambda_{1}-\alpha_{1} c_{1}(0)}.$$
\end{proposition}
\begin{proof}
In the derivation below we use this equality:
$$c_{1}=(A-\lambda I)^{-1}(A-\lambda I)c_{1}=(A-\lambda I)^{-1}(Ac_{1}-\lambda c_{1})=(A-\lambda I)^{-1}(\lambda_{1}-\lambda)c_{1}.$$
The characteristic polynomial for $A_{1}$ is as follows:
\begin{align*}
|A_{1}-\lambda I|=|A+\alpha_{1}c_{1}\delta-\lambda I|&=|A-\lambda I+\alpha_{1}c_{1}\delta|\\
&=|A-\lambda I|(1+\alpha_{1}\delta (A-\lambda I)^{-1}c_{1})\\
&=\left(\pm \prod_{i=1}^{\ell+1}(\lambda_{i}-\lambda)\right)\left(1+\frac{\alpha_{1}\delta c_{1}}{\lambda_{1}-\lambda}\right)\\
&=\left(\pm \prod_{i=2}^{\ell+1}(\lambda_{i}-\lambda)\right)\left(
\lambda_{1}+\alpha_{1}c_{1}(0)-\lambda \right).
\end{align*}
The roots of this polynomial are the eigenvalues of $A_{1}$ and equal  to $\{\lambda_{1}+\alpha_{1}c_{1}(0),\lambda_{2},\ldots,\lambda_{\ell+1}\}$, denoted as $\{\lambda^{1}_{1},\lambda^{1}_{2},\ldots,\lambda^{1}_{\ell+1}\}.$

 We denote the corresponding eigenvectors of $A_{1}$ by $\{c^{1}_{1},\ldots,c^{1}_{\ell+1}\}$ and they are computed as below. By definition, $(A_{1}-\lambda_{1}^{1}I)c^{1}_{1}=0$ and we derive:
 
 \begin{align*}
  (A_{1}-\lambda^{1}_{1}I)c^{1}_{1}&=\left(A+\alpha_{1}c_{1}\delta-(\lambda_{1}+\alpha_{1} c_{1})I\right)c^{1}_{1}\\
  &=(A-\lambda_{1}I)c^{1}_{1}+\alpha_{1}c_{1}\delta c^{1}_{1}-\alpha_{1}c^{1}_{1}\delta c_{1}\\
  &=0.
 \end{align*}
The solution of this system of equations is $c^{1}_{1}=c_{1}.$ This is the unique solution, since it solves a set of $\ell+1$ independent linear equations with $\ell+1$ variables.

Next, we show that for all $i\ge 2$, $c_{i}^{1}$ is a linear combination of $c_{i}$ and $c_{1}$, i.e.~$c_{i}^{1}=x_{i}c_{1}+c_{i}.$
We consider $(A_{1}-\lambda^{1}_{i}I)c^{1}_{i}=0.$ 

 \begin{align*}
  (A_{1}-\lambda^{1}_{i}I)c^{1}_{i}&=\left(A+\alpha_{1}c_{1}\delta-\lambda_{i}I\right)c^{1}_{i}\\
  &=(A-\lambda_{i}I)(x_{i}c_{1}+c_{i})+\alpha_{1}c_{1}(x_{i}c_{1}(0)+c_{i}(0))\\
  &=x_{i}Ac_{1}-x_{i}\lambda_{i}c_{1}+\alpha_{1}x_{i}c_{1}(0)c_{1}+\alpha_{1}c_{i}(0)c_{1}\\
  &=x_{i}(\lambda_{1}-\lambda_{i}+\alpha_{1}c_{1}(0))c_{1}+\alpha_{1}c_{i}(0)c_{1}\\
  &=0.
 \end{align*}
Thus choosing $x_{i}$ as below solves this equation:
$$x_{i}=\frac{-\alpha_{1}c_{i}(0)}{\lambda_{1}-\lambda_{i}+\alpha_{1}c_{1}(0)}=\frac{\alpha_{1}c_{i}(0)}{\lambda^{1}_{i}-\lambda^{1}_{1}}.$$
\end{proof}

\begin{cor}\label{cor:eigreal}
If the eigenvalues of an arbitrary matrix $A$ are real and $\alpha_{1}$ is real, then the eigenvalues of $A_{1}$ are real.
\end{cor}

 As is described above, $B$ is a combination of matrix $W$ and a finite number $M$ of well normed eigenvectors added to its first column. We want to use the result of Proposition~\ref{lem:eigvalue} to construct the eigenvalues of $B$ from those of $W$.
However, after the addition of an eigenvector, the other eigenvectors change. Adding another eigenvector $c_{i}$ of $W$ does not satisfy the conditions to use Proposition~\ref{lem:eigvalue} anymore, since $c_{1}$ is not an eigenvalue of the constructed matrix. The eigenvalues will change differently from the way described in that proposition.

To use Proposition~\ref{lem:eigvalue} iteratively, we define a sequence of matrices $W_{i}, i=1,2,\ldots,M,$ as follows:
\begin{equation}\label{eq:wi}
W_{i}:=W_{i-1}+\alpha_{i} c^{i-1}_{i},
\end{equation}
where $W_{0}=W.$
We will denote the eigenvalues and eigenvectors of $W_{i}$ as $\{\lambda^{i}_{j}\}_{j=1,\ldots,\ell+1}$ and $\{c^{i}_{j}\}_{j=1,\ldots,\ell+1}$. 
An intuitive but wrong way to construct $B$ would be to choose $\alpha_{i}=1$ for all $i$, as described above. The eigenvectors of $W_{i}$ dif{}fer from those of $W_{i-1}$, and are in fact linear combinations of those of $W_{i}.$ Thus we need to choose the factors $\alpha_{i}$ in such a way that for all $i$, the vector $c_{i}$ is completely added to $W$.
To ensure this, and thus that $B=W_{M}$ we add a fraction of vector $c^{i}_{i}$ in such a way that the `remaining' portion of that vector is spread out over the eigenvectors of $W_{i}$.

In other words, we are looking for scalars $\alpha_{i}$ such that the following set of equalities holds:
$$B=W+\sum_{j=1}^{M}c_{j}=\ldots=W_{i}+\sum_{j=i+1}^{M} c^{i}_{j}=\ldots=W_{M}.$$
Displaying $W_{i}$ in terms of $W$ by using Eq.~(\ref{eq:wi}) repeatedly, the dependence on $\alpha_{i}$ shows:

$$B=W+\sum_{j=1}^{M}c_{j}=\ldots=W+\sum_{j=1}^{i}\alpha_{j}c^{j-1}_{j}+\sum_{j=i+1}^{M}c^{i}_{j}=\ldots=W+\sum_{j=1}^{M}\alpha_{j}c^{j-1}_{j}.$$

This simplifies to the following for all $i=1,2,\ldots,M$:
$$\sum_{j=1}^{i-1}\alpha_{j}c^{j-1}_{j}+\sum_{j=i}^{M}c^{i-1}_{j}=\sum_{j=1}^{i}\alpha_{j}c^{j-1}_{j}+\sum_{j=i+1}^{M}c^{i}_{j}$$
and thus:
\begin{equation}\label{eq:alpha1}
(1-\alpha_{i})c^{i-1}_{i}=\sum_{j=i+1}^{M}(c^{i}_{j}-c^{i-1}_{j}).
\end{equation}

From Proposition~\ref{lem:eigvalue} and from the assumption we learned that $\lambda^{i}_{j}=\lambda^{i-1}_{j}=\ldots \lambda_{j}$ for all $j\ge i.$ The eigenvectors $c^{i}_{j}$ of $W_{i}$ are a linear combination of $c^{i-1}_{j}$ and $c^{i-1}_{i}$ as follows:

\begin{equation}\label{eq:cij}
c^{i}_{j}=c^{i-1}_{j}+\frac{\alpha_{i}c^{i-1}_{j}(0)}{\lambda_{j}-\lambda_{i}-\alpha_{i}c^{i-1}_{i}(0)}c^{i-1}_{i}.
\end{equation}

Combining Eq.~(\ref{eq:alpha1}) and (\ref{eq:cij}), we express $\alpha_{i}$ as a solution of the vector equations:
$$
(1-\alpha_{i})c^{i-1}_{i}=\sum_{j=i+1}^{M}\left(\frac{\alpha_{i}c^{i-1}_{j}(0)}{\lambda_{j}-\lambda_{i}-\alpha_{i}c^{i-1}_{i}(0)}\right)c^{i-1}_{i}.
$$

This set of dependent relations reduces to a single equation:
\begin{equation}\label{eq:alpha2}
1-\alpha_{i}=\alpha_{i}\sum_{j=i+1}^{M}\frac{c^{i-1}_{j}(0)}{\lambda_{j}-\lambda_{i}-\alpha_{i}c^{i-1}_{i}(0)}.
\end{equation}

This equation has at least one (possible non-real) solution, since $\lambda_{j}-\lambda_{i}\neq 0.$
We derive the following corollary regarding the eigenvalues of matrix $B$.

\begin{cor}\label{cor:eigvalues}
The eigenvalues $\lambda^{b}_{i}$ of $B$ are:
$$\lambda^{b}_{i}=\lambda_{i}+\alpha_{i}c^{i-1}_{i},$$
where $\lambda_{i}$ is an eigenvalue of $W,$ $c^{i-1}_{i}$ is constructed recursively via Equation~(\ref{eq:cij}) and $\alpha_{i}$ is the solution of Equation~(\ref{eq:alpha2}).
\end{cor}

Note that the eigenvalues of $B$ are not necessarily distinct, despite the fact that those of $W$ are.
Although we have assumed that we start with a tridiagonal matrix $W$ this is not necessary for the derivation above. 
Therefore, as long as the perturbation in the first column can be expressed as a linear combination of eigenvectors the procedure above works to find the eigenvalues of matrix $B$.
 
We derive from Corollary~\ref{cor:eigreal} and \ref{cor:eigvalues} that if all $\alpha_{i}$ are real, then the eigenvalues of $B$ are real.
Below we will provide a sufficient condition for this to be true. Therefore, without loss of generality we will reorder and rename the eigenvectors and eigenvalues such that $\lambda_{1}<\lambda_{2}<\ldots<\lambda_{M}<0$. 

We emphasize that the scalars $\alpha_{i}$ only depend on the first entries of eigenvectors $c_{i}$, cf.~Eq.~(\ref{eq:alpha2}). For notational convenience in the sequel, we rewrite Eq.~(\ref{eq:alpha2}) as a function of $y$:
\begin{equation}\label{eq:falpha}
f_{i}(y)=1-y-y\sum_{j=i+1}^{M}\frac{c^{i-1}_{j}(0)}{\lambda_{j}-\lambda_{i}-yc^{i-1}_{i}(0)}.
\end{equation}

We are interested in whether the function $f_{i}(\cdot)$ has a real root or not. If it has a real root, then we can use this root as a solution $\alpha_{i}$ to construct the eigenvalues of $W_{i}.$ By Proposition~\ref{lem:eigvalue} all these eigenvalues are real, if those of $W_{i-1}$ are real. In the proposition below we identify a sufficient condition for $f_{i}(\cdot)$ to have a real root.
We define $S_{i,j}:=(\lambda_{j}-\lambda_{i})/c_{i}^{i-1}(0)$, the singularities of $f_{i}(\cdot)$. Since $\lambda_{j+1}>\lambda_{j}$ for all $j$ this sequence is strictly increasing and positive (decreasing and negative) in $j$ when $c_{i}^{i-1}(0)>0 (<0)$. As is reasoned before, $c_{i}^{i-1}(0)\neq 0.$

\begin{proposition}\label{lem:cijreal}
Suppose that $c^{i-1}_{i}(0)$ and $c^{i-1}_{i+1}(0)$ are both positive.
Then $f_{i}(\cdot)$ has at least one real positive root. 
\end{proposition}
\begin{proof}
By the assumption in the proposition, the quotient $\frac{yc_{i+1}^{i-1}(0)}{\lambda_{i+1}-\lambda_{i}-yc_{i}^{i-1}(0)}>0$ when $0<y<S_{i,i+1},$ since $\lambda_{i+1}-\lambda_{i}>yc_{i}^{i-1}(0)$ and $c_{i+1}^{i-1}(0)>0.$
This quotient is the dominant term near the singularity $S_{i,i+1}$ and thus $\lim_{y\uparrow S_{i,i+1}} f_{i}(y) =-\infty.$ 
Since $f_{i}(0)=1,$ there is at least one positive real root on the interval $(0,S_{i,i+1})$.
%
\end{proof}

\begin{proposition}\label{lem:cijnegreal}
When $c^{i-1}_{i}(0)<0$ there is at least one real positive root of $f_{i}(\cdot).$
\end{proposition}
\begin{proof}
Recall that $S_{i,i+1}$ is the largest singularity and smaller than zero when $c^{i-1}_{i}(0)<0$. Therefore, $f_{i}(\cdot)$ is continuous on $[0,\infty)$. Since $f_{i}(0)=1$ and $\lim_{y\rightarrow\infty}f_{i}(y)=-\infty$ the function $f_{i}(\cdot)$ has at least one positive real root.
%
\end{proof}

In the proposition below we will provide a condition that ensures that all scalars $\alpha_{i}$ are real.

\begin{proposition}\label{lem:alphareal}
All scalars $\alpha_{i}$ are real when $c_{j}(0)<0$ for all $j\le I$ and $c_{j}(0)<0$ for all $j> I$ with $I\in\{0,1,2,\ldots,M\}.$
\end{proposition}
\begin{proof}
With this sequence of first elements, either $c_{1}(0)<0$ or $c_{1}(0),c_{2}(0)>0$.
We know by Proposition~\ref{lem:cijreal} and Proposition~\ref{lem:cijnegreal}  that there is at least one real root for $f_{0}(\cdot)$, and moreover that it is positive.
Choose $\alpha_{1}$ to be the smallest positive real root of $f_{i}(\cdot)$.

Eq.~(\ref{eq:cij}) implies the following recursion for the first element of $c^{1}_{j}$:
$$c^{1}_{j}(0)=c_{j}(0)\left(1+\frac{\alpha_{1}c_{1}(0)}{\lambda_{j}-\lambda_{1}-\alpha_{1}c_{1}(0)}\right)=c_{j}(0)\left(\frac{\lambda_{j}-\lambda_{1}}{\lambda_{j}-\lambda_{1}-\alpha_{1}c_{1}(0)}\right).
$$
If $c_{1}(0)<0,$ then $\frac{\lambda_{j}-\lambda_{1}}{\lambda_{j}-\lambda_{1}-\alpha_{1}c_{1}(0)}>1,$ since $\alpha_{1}>0$. 

If $c_{1}(0)>0,$ then $\frac{\lambda_{j}-\lambda_{1}}{\lambda_{j}-\lambda_{1}-\alpha_{1}c_{1}(0)}>0,$ since $\alpha_{1}\in (0,S_{1,2})$.

Thus $c_{j}^{1}(0)$ has the same sign as $c_{j}(0)$ for all $j=2,3,\ldots,M.$ The provided structure of this elements ensures that 
$f_{1}(\cdot)$ has at least one positive real root $\alpha_{1}$, since now $c^{1}_{2}(0)<0$ or $c^{1}_{2}(0),c^{1}_{3}(0)>0$. We can repeat this argument and construct all positive real roots $\alpha_{i}$ for all $i$.
\end{proof}
 
We conjecture that the sequence $c_{i}(0)$ satisfies the condition, given in Proposition~\ref{lem:alphareal} for the case when $B$ is sub-stochastic matrix where the row sum is negative only in the first row. Numerical examples confirm this conjecture, but  we leave its proof as an open problem.

\section{Applications}\label{sec:applications}
Matrix $B$ discussed in the previous sections occurs in many models, particularly when considering a 2 dimensional Markov chain, using the successive lumping approach, cf.~\cite{sslqsf2013}. In order to find the steady state distribution one can use this approach to compute the rate matrix $R$. In many cases calculating this matrix using successive lumping requires the inversion of a matrix with the structure of $B$. Examples can be found in queueing (an $E_{k}/M/c$ queue with batch service), reliability systems and inventory models (a inventory model with batch arrivals and random lead time). Specific about these systems an the complete procedure to compute the steady state distribution can be found in \cite{Katehakis2015comparative}.
Below we will discuss other (classes of) applications in which this type of matrix arises naturally or can be constructed.
\subsection{A general non-transient Markov chain}
Let $Q$ be the transition rate matrix of a general - non-transient Markov chain and let $B$ be the matrix 
with elements 
$$b(i,j)=
\begin{dcases}
q(0,0)-1 & \text{if} \  (i,j) =(0,0), \\
q(i,j)  &\text{otherwise},\\
\end{dcases}
$$
 i.e., $B=Q-\delta' \cdot \delta$.
The solution to the steady state equations:
\begin{align*}
\pi Q &=0 \\
\pi 1 &=1 \\
\end{align*}
is given by: $\pi=\displaystyle\frac{\delta B^{-1}}{\delta B^{-1}\mathbf{1}}.$ Algorithm \ref{Alg:Invd}, can be used to obtain the solution efficiently and in case of a homogenous process Algorithm~\ref{Alg:HomInv} should be applied.

\subsection{A birth-and-death process with an absorbing state}

In this section, we consider a special case of the structure described above that has applications to a birth-and-death process with abandonments. These models occur for example in diffusion processes (cf.~\cite{janssen1981nonequilibrium}) and in  randomly changing environments (cf.~\cite{gaver1984finite}). We assume that in every state there is an (equal) positive rate to leave the system and go to an absorbing state (state $0$). The remaining states form a birth-and-death process. This process is element homogenous according to its description in Definition \ref{def:eh}, but with the difference that $b^{u}_{0}=0$ and $b^{d}_{1}=0.$
Therefore the transition matrix has the following form:
\begin{equation}\label{eq:bdprnh}
B= \left[\begin{array}{rccccc}
-b^d&0&0&0&0&\cdots\\ 
b^z&-b^z-b^{u}&b^u&0&0&\cdots\\ 
b^z&b^d&-b^w&b^u&0&\ddots\\
b^z&0&b^d&-b^w&b^u&\ddots\\
b^z&0&0& b^d&-b^w&\ddots\\
\vdots&\vdots&\ddots&\ddots&\ddots&\ddots
\end{array}\right],
\end{equation}
where options: 1: $b^z_i=b^z$ for all $i$  or the original $B$ with $b^u_0=0$ 


The procedure to find the inverse of $B$ is described in Algorithm~\ref{Alg:InvBD} below. This is a special case of Algorithm~\ref{Alg:HomInv}, that requires element homogeneity; only the first row of this matrix slightly does not meet this requirement.

\begin{algorithm}{}\label{Alg:InvBD}
\begin{itemize}
\item[]
At stage 1: 
\begin{enumerate}[label=\alph*)]
\item
Calculate $\gamma$ using Eq. (\ref{eq:gammah}).
\item
Calculate $\psi$ using \eqqref{eq:psi}.
\item
The $0^{th}$  column of $C$ is computed  using  Eq.~(\ref{eq:c1})  of Proposition \ref{prop:firstcol}. 
\item  The remaining elements of the $0^{th}$  row of  $C$ are $0$, as is described in Proposition~\ref{prop:BDprocess1}.
\item Calculate element $c(1,1)$ by Proposition~\ref{prop:BDprocess2}.
\item The remaining elements ($j\geq 2$) of the first row are: $c(1,j)=\gamma^{j-1} c(1,1)$.
\end{enumerate}
\item[]
At stage $i= 1,2,\ldots,\ell+1$\ : 
\begin{enumerate}[label=\alph*)]
\item The diagonal element $c(i,i)$ is computed by \eqqref{eq:diagonal}.
\item The elements of $C_{i}$, the $i^{th}$ row of $C$ to the right of the diagonal  are computed using Proposition \ref{prop:hom} and in particular \eqqref{eq:gamchom} in that proposition.
\item The under diagonal elements of $C'_{i},$ the $i^{th}$ column of $C$ are computed using
Eq. (\ref{eq:cijunder}) of Proposition~\ref{prop:cijunder}.
\end{enumerate}
\end{itemize}
\end{algorithm}

Most of the steps above are similar to those in Algorithm~\ref{Alg:HomInv}, and the two Propositions below justify the remaining steps. 
\begin{proposition}\label{prop:BDprocess1}
The elements $c(0,j)$ are equal to zero for all $j\geq 1.$
\end{proposition}
\begin{proof}
This result follows immediately from the fact that all elements are nonpositive and by considering $B_{1}C'_{0}=b^{d}(-1/b^{d})+z\sum^{\infty}_{j=1} c(0,j)=1$.
\end{proof}

\begin{proposition}\label{prop:BDprocess2}
The element $c(1,1)$ can be calculated as follows:
$$c(1,1)=\frac{1}{-b^{z}-b^{u}+b^{d}\gamma}.$$
\end{proposition}
\begin{proof}
We consider the equation $C_{1}B'_{1}=1.$ This gives us:
$$(-b^{z}-b^{u})c(1,1)+b^{d}c(1,2)=1.$$
Because $c(1,2)=\gamma c(1,1)$ we derive:
$$c(1,1)=\frac{1}{(-b^{z}-b^{u})+b^{d}\gamma}.$$
\end{proof}

\subsection{Value Functions}
In many models,  e.g., to compute the value function of the 
expected  $\alpha-$discounted cost associated with a continuous time Markov chain, one needs to
to solve equations of the following type:
\begin{equation}\label{eq:vd}
\alpha V =c+Q V
\end{equation} 
where $c$ is  the cost rate function defined on the state space and $V$ the unknown value function to 
be determined as a solution of Eq. (\ref{eq:vd}).
We refer to  \cite{ross2014dp}, \cite{ross2013app}  for more background on these models.
The solution to this equation is 
\begin{equation}\label{eq:vdsol}
 V =-Q_{\alpha}^{-1} c
\end{equation} 
where $Q_{\alpha}=(Q-\alpha I)$ is a rate matrix of a transient Markov Process where in addition to the rates specified by $Q$ one has introduced an additional event of `exiting' or `halting' the process at a
rate $\alpha$.  

When $Q$ has a tridiagonal form, i.e.~is a birth-and-death process, we can use the algorithm of the previous section to compute this inverse of matrix $Q_{\alpha}.$
 
\subsection*{ACKNOWLEDGMENT} 
We gratefully acknowledge support for this project from the National Science Foundation (NSF grant CMMI-14-50743).


\bibliographystyle{plain} 
\bibliography{InverseRef}

\begin{thebibliography}{10}

\bibitem{alexanderian2013continuous}
A.~Alexanderian.
\newblock On continuous dependence of roots of polynomials on coefficients.
\newblock {\em Retrieved from
  http://users.ices.utexas.edu/$\sim$alen/articles/polyroots.pdf}, 2013.

\bibitem{ammar1996classical}
G.S. Ammar.
\newblock Classical foundations of algorithms for solving positive definite
  {T}oeplitz equations.
\newblock {\em Calcolo}, 33(1-2):99--113, 1996.

\bibitem{ammar1988superfast}
G.S. Ammar and W.B. Gragg.
\newblock Superfast solution of real positive definite {T}oeplitz systems.
\newblock {\em SIAM Journal on Matrix Analysis and Applications}, 9(1):61--76,
  1988.

\bibitem{barrett1981inverses}
W.W. Barrett and P.J. Feinsilver.
\newblock Inverses of banded matrices.
\newblock {\em Linear Algebra and its Applications}, 41:111--130, 1981.

\bibitem{ben2003g}
Adi Ben-Israel and Thomas~NE Greville.
\newblock {\em Generalized inverses: theory and applications}, volume~15.
\newblock Springer Science \& Business Media, 2003.

\bibitem{gaver1984finite}
D.P. Gaver, P.A. Jacobs, and G.~Latouche.
\newblock Finite birth-and-death models in randomly changing environments.
\newblock {\em Advances in applied probability}, pages 715--731, 1984.

\bibitem{grassmann2002real}
W.K Grassmann.
\newblock Real eigenvalues of certain tridiagonal matrix polynomials, with
  queueing applications.
\newblock {\em Linear Algebra and its Applications}, 342(1):93--106, 2002.

\bibitem{heinig1984algebraic}
G.~Heinig and K.~Rost.
\newblock {\em Algebraic methods for {T}oeplitz-like matrices and operators}.
\newblock Springer, Basel, Switzerland, 1984.

\bibitem{ikebe1979inverses}
Y.~Ikebe.
\newblock On inverses of {H}essenberg matrices.
\newblock {\em Linear Algebra and its Applications}, 24:93--97, 1979.

\bibitem{janssen1981nonequilibrium}
H.-K. Janssen.
\newblock On the nonequilibrium phase transition in reaction-diffusion systems
  with an absorbing stationary state.
\newblock {\em Zeitschrift f{\"u}r Physik B Condensed Matter}, 42(2):151--154,
  1981.

\bibitem{Katehakis2015comparative}
M.N Katehakis, L.C. Smit, and F.M. Spieksma.
\newblock A comparative analysis of the successive lumping and the lattice path
  counting algorithms.
\newblock {\em Journal of Applied Probability, to appear}, 2015.

\bibitem{sslqsf2013}
M.N. Katehakis, L.C. Smit, and F.M. Spieksma.
\newblock {DES} and {RES} processes and their explicit solutions.
\newblock {\em Probability in the Engineering and Informational Sciences},
  29:191--217, 2015.

\bibitem{kilicc2013inverse}
E.~K{\i}l{\i}{\c{c}} and P.~Stanica.
\newblock The inverse of banded matrices.
\newblock {\em Journal of Computational and Applied Mathematics},
  237(1):126--135, 2013.

\bibitem{li2010inverses}
H.-B. Li, T.-Z. Huang, X.-P. Liu, and H.~Li.
\newblock On the inverses of general tridiagonal matrices.
\newblock {\em Linear Algebra and Its Applications}, 433(5):965--983, 2010.

\bibitem{mallik2001inverse}
R.K. Mallik.
\newblock The inverse of a tridiagonal matrix.
\newblock {\em Linear Algebra and its Applications}, 325(1):109--139, 2001.

\bibitem{martinsson2005fast}
P.-G. Martinsson, V.~Rokhlin, and M.~Tygert.
\newblock A fast algorithm for the inversion of general {T}oeplitz matrices.
\newblock {\em Computers \& Mathematics with Applications}, 50(5):741--752,
  2005.

\bibitem{meurant1992review}
G.~Meurant.
\newblock A review on the inverse of symmetric tridiagonal and block
  tridiagonal matrices.
\newblock {\em SIAM Journal on Matrix Analysis and Applications},
  13(3):707--728, 1992.

\bibitem{ross2013app}
Sheldon~M Ross.
\newblock {\em Applied probability models with optimization applications}.
\newblock Courier Corporation, 2013.

\bibitem{ross2014dp}
Sheldon~M Ross.
\newblock {\em Introduction to stochastic dynamic programming}.
\newblock Academic press, 2014.

\end{thebibliography}

\end{document}